\documentclass[12pt]{amsart}
\usepackage{xcolor}
\usepackage{amsmath}
\usepackage{amssymb}
\usepackage{amscd}
\newtheorem{thm}{Theorem}[section]
\newtheorem{cor}[thm]{Corollary}
\newtheorem{lem}[thm]{Lemma}
\newtheorem{prop}[thm]{Proposition}
\theoremstyle{definition}
\newtheorem{defn}[thm]{Definition}
\newtheorem{ex}[thm]{Example}
\theoremstyle{remark}

\newtheorem{prob}[thm]{Problem}
\numberwithin{equation}{section}
\DeclareMathOperator{\N}{\mathbb{N}}
\DeclareMathOperator{\Z}{\mathbb{Z}}
\def\Q{\mathbb Q}

\begin{document}

\author[P.V. Danchev]{Peter V. Danchev}
\address{Institute of Mathematics and Informatics, Bulgarian Academy of Sciences \\ "Acad. G. Bonchev" str., bl. 8, 1113 Sofia, Bulgaria.}
\email{danchev@math.bas.bg; pvdanchev@yahoo.com}
\author[P.W. Keef]{Patrick W. Keef}
\address{Department of Mathematics, Whitman College, 345 Boyer Avenue, Walla Walla, WA, 99362, United States of America.}
\email{keef@whitman.edu}

\title[On Some Versions of Hopficity for Abelian Groups]{On Some Versions of Hopficity for \\ Abelian Groups}

\maketitle

\begin{abstract} We completely describe in certain important cases the class of commutative {\it co-finitely Hopfian} groups as defined by Bridson-Groves-Hillman-Martin in the journal Groups, Geometry, and Dynamics on 2010 (see \cite{BGHM}). We also consider and give a satisfactory description of several related classes of commutative groups.

We also discuss in the commutative case a slightly more general version of co-finitely Hopfian groups called {\it almost co-finitely Hopfian} groups, as well as a more general version of Hopfian groups called {\it almost finitely Hopfian} groups.
\end{abstract}

\section{Introduction and Fundamentals}

As is well-known, an arbitrary ({\it not} necessarily commutative) group is said to be {\it Hopfian} if every surjective endomorphism (i.e., epimorphism) is an automorphism. Dually, a group is said to be {\it co-Hopfian} if every injective endomorphism (i.e., monomorphism) is an automorphism. These properties arose quite naturally in Hopf's work on self-maps of surfaces (see \cite{H}).

Inspired by these two fundamental algebraic objects and some topological considerations which were about compact manifolds, and especially that the fundamental groups of compact manifolds are finitely presented, all of this focuses attention on such groups motivating Bridson et al. to define in \cite{BGHM} a group to be {\it co-finitely Hopfian} if every endomorphism whose image is of finite index is an automorphism, i.e., for any endomorphism $\phi: G\to G$ with finite $G/\phi(G)$, $\phi$ is an automorphism; that is, both ker$\phi=\{0\}$ and coker$\phi=\{0\}$. This condition is definitely stronger than the Hopfian property, as a non-trivial finite group will always be Hopfian, but never be co-finitely Hopfian. They proved some important results on these groups in the non-commutative case, one of which leads even to a rather more general property (see \cite[Theorem 3.1]{BGHM}).

Attempting to strengthen this basic concept, we call a group $G$ {\it co-finitely injective} provided that, for any endomorphism $\phi: G\to G$ with finite $G/\phi(G)$, $\phi$ is a monomorphism (that is, ker$\phi=\{0\}$). Again, any group that is co-finitely injective will be Hopfian, but the converse fails in general. Dually, we call a group $G$ {\it co-finitely surjective} provided that, for any endomorphism $\phi: G\to G$ with finite $G/\phi(G)$, $\phi$ is an epimorphism (that is, co-ker$\phi=\{0\}$). It is obvious that a group is co-finitely Hopfian if, and only if, it is both co-finitely injective and co-finitely surjective. In fact, for our purposes, it would have been better to refer to a ``co-finitely Hopfian group" as {\it co-finitely bijective}, but however for completeness of the exposition we will preserve the terminology of \cite{BGHM}.

From here on, the word ``group" will signify an additively written abelian group; notations used in the text are standard being mainly in agreement with \cite{F1,F2}. Everywhere in the text, the word ``summand" means ``direct summand".

We begin by showing that in studying these classes of groups there is little loss of generality in restricting to the case of reduced torsion-free groups. In the latter case, we completely describe these classes for groups of finite rank, as well as those that are completely decomposable or algebraically compact. We then show by a concrete example that these characterizations do {\bf not} necessarily extend to broader classes of groups, such as the so-called {\it Butler groups} of infinite rank, also known in the existing literature as {\it $B_2$-groups}.

Likewise, we characterize in some aspect the so-termed {\it almost co-finitely Hopfian groups} that slightly expand the defined above co-finitely Hopfian groups (see Proposition~\ref{torsion} through Corollary~\ref{Bassian}). Also, we classify in some way the so-named {\it almost finitely Hopfian groups} which generalize the classical Hopfian group (see Theorem~\ref{sep}). And finally, some of what we do {\it not} achieve in the present article is explicitly posed as two challenging queries (see Problems~\ref{8} and \ref{9}, respectively).

\section{Preliminaries}

We will frequently use the following observation, which is easy to verify and so stated without proof.

\begin{lem}\label{0.5}
Suppose $G$ is a group and $A$ is a summand of $G$. If $G$ is co-finitely injective, co-finitely surjective, or co-finitely Hopfian, then $A$ satisfies the same property.
\end{lem}

The following examples can be easily verified by considering the endomorphism given by multiplication by $p$, so we leave it with no proof.

\begin{ex}\label{0.4}
If $p$ is a prime and $n\in N$, then $\Z(p^n)$ is neither co-finitely injective nor co-finitely surjective. On the other hand, $\Z(p^\infty)$ is co-finitely surjective, but {\it not} co-finitely injective.
\end{ex}

Since a group is torsion-free if, and only if, it does not have a summand isomorphic to one of the groups mentioned above, the following statement is an immediate consequence.

\begin{cor}\label{0.25}
Suppose $G$ is a group.

(a) If $G$ is co-finitely injective, then $G$ is torsion-free.

(b) If $G$ is co-finitely surjective, then the reduced part of $G$ is torsion-free.
\end{cor}

We begin by characterizing these classes in the very tractable case of divisible groups.

\begin{prop}\label{1.9} Suppose $D$ is a divisible group.

(a) $D$ is co-finitely injective if and only if it is torsion-free of finite rank;

(b) $D$ is always co-finitely surjective.
\end{prop}

\begin{proof}
Regarding (a), consider necessity. By Corollary~\ref{0.25}, $D$ must be torsion-free. Next, if $D$ had infinite rank, then $D\cong \Q\oplus D$, so that $D$ is not co-finitely injective, as required.

Turning to sufficiency, suppose $D\cong \Q^n$ and $\phi:D\to D$ is an endomorphism with $D/\phi(D)$ finite. It follows that $\phi(D)$ is also torsion-free divisible of rank $n$, so that $\phi$ must be an isomorphism (note we have actually shown that such a $D$ is co-finitely Hopfian).

As to (b), suppose again that $\phi:D\to D$ is an endomorphism with $D/\phi(D)$ finite. It follows that $D/\phi(D)$ is a finite divisible group, i.e., it must be $\{0\}$. Therefore, $D=\phi(D)$, so that $\phi$ is surjective, as needed.
\end{proof}

We now discuss what happens when we break a group down into a sum of a divisible group and a reduced group.

\begin{prop}\label{1.95} Suppose $G$ is a group with $G=R\oplus D$, where $R$ is reduced and $D$ is divisible.

(a) $G$ is co-finitely injective if and only if $R$ is co-finitely injective and $D\cong \Q^n$ for some $n<\omega$.

(b) $G$ is co-finitely surjective if and only if $R$ is co-finitely surjective.
\end{prop}

\begin{proof} For both parts, necessity follows directly from Lemma~\ref{0.5} and Proposition~\ref{1.9}.

Turning to sufficiency, let $\phi:G\to G$ be an endomorphism with $G/\phi(G)$ finite. In either (a) or (b), $\phi(D)\leq D$, so that $\phi$ induces a homomorphism $$\overline \phi:R\cong G/D\to G/D\cong R.$$ It is readily seen that $(G/D)/\overline \phi(G/D)$ is naturally an epimorphic image of $G/\phi(G)$, so that $\overline\phi(G/D)$ also has finite index in $G/D$.

Regarding (a), since $R$ is assumed to be co-finitely injective, it follows that $\overline \phi$ is injective. This readily forces that $\phi(G)\cap D=\phi(D)$. Therefore, $$D/\phi(D)=D/(\phi(G)\cap D)\leq G/\phi(G)$$ is also finite. But, since $D$ is co-finitely injective, we can conclude that $\phi$ is injective when restricted to $D$. This, and the injectivity of $\overline \phi$ yields that $\phi$ itself is injective.

Regarding (b), since $R$ is assumed to be co-finitely surjective, it follows that $\overline \phi$ is surjective. Again, since $D/(\phi(G)\cap D)$ embeds in  $G/\phi(G)$, we can infer that it is finite. But, since $D/(\phi(G)\cap D)$ is divisible, this can only hold if it is $\{0\}$; i.e., $D\leq \phi(G)$. This, combined with the surjectivity of $\overline \phi$, guarantee that $\phi$ itself is surjective.
\end{proof}

So, a co-finitely injective group, as well as a reduced co-finitely surjective group, will always be torsion free. This leads to the following assertion.

\begin{cor}\label{0.9}
The group $G$ is co-finitely Hopfian if and only $G\cong R\oplus \Q^n$, where $n<\omega$ and $R$ is reduced co-finitely Hopfian (so that $G$ must be torsion-free).
\end{cor}

Thus, in describing the groups in these three classes, there is no loss of generality in assuming that $G$ is torsion-free and reduced. We now begin a more detailed discussion of these conditions in the plainly handled case of torsion-free groups of finite rank by proving the following.

\begin{prop}\label{1} If $G$ is any torsion-free group of finite rank, then $G$ is co-finitely injective (and hence Hopfian).
\end{prop}

\begin{proof} Suppose $\phi:G\to G$ is an endomorphism. If $G/\phi(G)$ is finite, then one sees that both $\phi(G)$ and $G$ have the same rank, so the kernel of $\phi$ is exactly $\{0\}$, as wanted.
\end{proof}

On the other hand, we have the following useful criterion.

\begin{prop}\label{3} If $G$ is a torsion-free group of finite rank, then the following three assertions are equivalent:

(a) $G$ is co-finitely Hopfian;

(b) $G$ is co-finitely surjective;

(c) $G$ is co-Hopfian;

(d) $G$ is divisible.
\end{prop}

\begin{proof} That (a) implies (b) is immediate from the definitions. Suppose next that (b) holds, i.e.,  $G$ is co-finitely surjective. To verify (c), let $\phi:G\to G$ be an injective homomorphism. Since $G$ is torsion-free of finite rank and $G\cong \phi(G)\leq G$, it is well known that $G/\phi(G)$ must be finite. So, if $G$ is co-finitely surjective, $\phi$ must be surjective, i.e., $G$ is co-Hopfian.

Suppose now that (c) holds. It elementarily follows that any co-Hopfian torsion-free group must be divisible; so, in particular, this holds when $G$ has finite rank (this follows by considering the injective endomorphisms $\phi(x)=px$ for all primes $p$.)

Finally, (d) implies (a) follows from Proposition~\ref{1.9}, thus completing the proof.
\end{proof}

\begin{ex}
Any torsion-free group of finite rank that fails to be divisible is co-finitely injective, but {\it not} co-finitely surjective, and hence {\it not} co-finitely Hopfian.
\end{ex}

On the other hand, the following is immediate from Proposition~\ref{3}.

\begin{cor}
A finite rank torsion-free group that is co-finitely surjective will be co-finitely injective.
\end{cor}

Furthermore, to show these ideas are independent for groups of infinite rank, we present some additional examples.

\begin{ex}
Suppose $G$ is the $p$-adic integers for some arbitrary but a fixed prime $p$. Note that any non-zero endomorphism $\phi:G\to G$ will be multiplication by some $x\in G$ and, in particular, it will be injective. This immediately insures that $G$ must be co-finitely injective. Moreover, if $\phi:G\to G$ is multiplication by $p$, then it must be that $G/\phi(G)\cong \Z(p)$. Since this quotient is finite, but $G\not= \phi(G)$, it follows that $G$ cannot be co-finitely surjective.

On the other hand, if $G$ is a torsion-free divisible group of infinite rank, it follows that $G$ is co-finitely surjective, but not co-finitely injective, as we asked.
\end{ex}

At a later point, we will exhibit examples of reduced groups that are co-finitely surjective, but {\it not} co-finitely injective.

\section{Completely decomposable torsion-free groups}

The following slight generalization of Proposition~\ref{3} applies, for example, to torsion-free groups that are completely decomposable.

\begin{cor}\label{exclud}
Suppose $G=\bigoplus_{i\in I} A_i$, where each $A_i$ is torsion-free of finite rank. Then, $G$ is co-finitely Hopfian if and only if it is co-Hopfian, i.e., if and only if it is finite-rank and divisible. And $G$ is co-finitely surjective if and only if it is divisible.
\end{cor}

\begin{proof}
Suppose $G$ is either co-finitely Hopfian, co-Hopfian, or co-finitely surjective. Since a summand of such a group certainly retains that property, we can conclude that each $A_i$ satisfies that property. It follows from Proposition~\ref{3} that each $A_i$ is divisible, so that $G$ is divisible as well. The rest of the argument follows directly from Proposition~\ref{1.9}.
\end{proof}

Suppose $G=\bigoplus_{i\in I}A_i$ is a completely decomposable torsion-free group, where each $A_i$ has rank 1 and type $\tau_i$. We will say $G$ satisfies the {\it descending type condition on rank 1 summands} if there does {\it not} exist distinct elements of $I$, say $i_1,i_2,\dots$, such that $\tau_1\geq \tau_2\geq \tau_3\geq  \cdots$.

\medskip

The following claim clarifies this condition a bit more.

\begin{lem}\label{descending}
The above completely decomposable torsion-free group $G=\bigoplus_{i\in I}A_i$ satisfies the descending type condition  on rank 1 summands if and only if two things hold:

(a) For any $i\in I$, the set of all $j\in I$ such that $\tau_j=\tau_i$ is finite.

(b) For any non-empty subset $J\leq I$, $\{\tau_j:j\in J\}$ has a minimal element.

\end{lem}

\begin{proof}
Suppose (a) and (b) hold; we assume $i_1,i_2,\dots$ are distinct elements of $I$ such that $\tau_1\geq \tau_2\geq \tau_2\geq  \cdots$, and derived a contradiction. If an infinite number of these inequalities were strict, it would follow that $J=\{i_k:k\in \N\}$ violates (b). On the other hand, if $N\in \N$ can be found such that $\tau_N=\tau_{N+1}= \tau_{N+2}=  \cdots$, then this would violate condition (a).

The converse is similarly elementary and thus left to the reader.
\end{proof}

We now come to the promised above criterion for a completely decomposable torsion-free group to be co-finitely injective and, by extension, co-finitely Hopfian. Specifically, the following is true.

\begin{thm}\label{4} Suppose $G=\bigoplus_{i\in I}A_i$ is a completely decomposable torsion-free group, where each $A_i$ has rank 1 and type $\tau_i$. Then, the following statements are equivalent:

(a) $G$ is co-finitely injective;

(b) $G$ is Hopfian;

(c) $G$ satisfies the descending type condition on rank 1 summands.
\end{thm}

\begin{proof} Certainly, (a) assures (b). To establish that (b) ensures (c), we argue by contrapositive. So, suppose (c) fails for $G$; we need to check that $G$ is not Hopfian. To this target, suppose $i_1,i_2,\dots$ are distinct elements of $I$ such that $\tau_1\geq \tau_2\geq \tau_3\geq  \cdots$. Considering the summand
$
                     H= \bigoplus_{j\in \N} A_{i_j},
$
if $H$ is not Hopfian, then $G$ is also not Hopfian. Replacing $G$ by $H$, there is no loss of generality in assuming $I=\N$ and, for every $i\in \N$, $\tau_i\geq \tau_{i+1}$.

Let $\{S_k\}_{k\in \N}$ be a collection of disjoint and infinite subsets of $\N$. Replacing $S_k$ by $\{j\in S_k: j\geq k\}$, there is no loss of generality in assuming that $i\geq k$ for all $i\in S_k$.

Let $\{a_i\}_{i\geq k}$ be an enumeration of $A_k$. It is readily seen that, for every $k, i\in \N$ with $k\leq i$,  there is a homomorphism $\gamma:A_i\to A_k$ such that $a_i\in \gamma(A_i)$. This gives that, for every $k\in \N$, there is a short-exact sequence
$$
           0 \to K_k\to \bigoplus_{i\in S_k} A_i \to A_k \to 0.
$$
Let $\phi_k$ be the right-hand map of this sequence. Putting these together, for all $k\in \N$, we get a surjective composition
$$
         \phi:  G\to \bigoplus_{k\in \N} \left(\bigoplus_{i\in S_k} A_i\right) \to \bigoplus_{k\in \N} A_k=G.
$$
Since each $K_k\ne \{0\}$, the kernel of $\phi$ is non-zero and, in fact, infinite. Therefore, $G$ is not Hopfian, as stated.

\medskip 

We now prove that if $G$ satisfies (c), then (a) must hold. In other words, if $\{\tau_i:i\in I\}$ satisfies the descending type condition, then $G$ is co-finitely injective. So, suppose $\phi:G\to G$ is a homomorphism such that $G/\phi(G)$ is finite. We need to prove that such a map $\phi$ is injective.

Let $\{\tau_j\}_{j\in J}$ be the collection of {\it distinct} types among the $\{\tau_i\}_{i\in I}$. Thus,
$$G=  \bigoplus_{i\in I}A_i=\bigoplus_{j\in J}B_j,$$
where, for each $j\in J$, $B_j$ is $\tau_j$-homogeneous completely decomposable. Note that, in virtue of Lemma~\ref{descending}(a), each $B_j$ must have finite rank. Since every subset of $\{\tau_j\}_{j\in J}$ has a minimal element thanks to Lemma~\ref{descending}(b), it follows that we can well-order $J$ so that, for each $\alpha\in J$,
$
                  \tau_\alpha
$ is minimal in $\{\tau_\beta: \alpha\leq \beta\in J\}$. Let $Q=\Q\otimes G$ and, for every $\alpha\in J$, let
$$
                      G_\alpha = \bigoplus_{\alpha\leq \beta\in J} B_\beta\leq G
$$
and $Q_\alpha=\Q\otimes G_\alpha\leq Q$. We will also interpret $Q_\alpha=\{0\}$ for all sufficiently large $\alpha$.

Clearly, $\phi:G\to G$ extends to a homomorphism $\gamma:Q\to Q$, and since $G/\phi(G)$ is finite, and hence torsion, it follows that $\gamma$ must be surjective.

We claim, for all $\alpha\in J$, that $\gamma(Q_\alpha)\leq Q_\alpha$: For every $\lambda<\alpha$ and $\sigma\geq \alpha$, let $\pi_\lambda:G\to B_\lambda$ be the usual projections and $\mu_\sigma: B_\sigma \to G$ be the usual inclusions. Since $\tau_\lambda$ is minimal in $\{\tau_\beta: \beta \geq \lambda\}$ and $\lambda< \alpha\leq \sigma$, we can deduce that $\tau_\lambda \not> \tau_\sigma$. And since $\sigma\ne \lambda$, we also have $\tau_\lambda \ne \tau_\sigma$, i.e, $\tau_\lambda \not\geq \tau_\sigma$. Therefore, $$\pi_\lambda\circ \phi\circ \mu_\sigma:A_\sigma\to A_\lambda$$ must be $\{0\}$. Letting $\lambda$ range over all values $<\alpha$ and $\sigma$ range over values $\geq \alpha$, gives the wanted result.

For each $\alpha\in J$, let $\gamma_\alpha: Q_\alpha\to Q_\alpha$ be the obvious restriction of $\gamma$, and $\gamma^\alpha:Q/Q_\alpha\to Q/Q_\alpha$ be the natural homomorphism induced by $\gamma$. Since $\gamma$ is surjective, it immediately follows that each $\gamma^\alpha$ is surjective as well.

We, actually, claim that each $\gamma^\alpha$ is also injective, which of course is only possible if $\gamma$ itself is injective. We handle this by inducting on $\alpha\in J$, so we assume $\gamma^\sigma$ is injective for all $\sigma<\alpha$ and argue that $\gamma^\alpha$ must also be injective: Suppose first that $\alpha$ is a limit ordinal. If $\overline{0}\ne x\in Q/Q_\alpha$, then there clearly must exist an ordinal $\sigma<\alpha$ such that $x$ is also non-zero in the factor-group $Q/Q_\sigma$. Since we are assuming $\gamma^\sigma$ is injective, it follows that $\gamma(x)$ represents a non-zero element of $Q/Q_\sigma$. In other words, $\gamma(x)$ is not in $Q_\sigma$, so that it cannot be in $Q_\alpha$, either. Therefore, $\gamma^\alpha(x)\ne 0$. So, the kernel of $\gamma^\alpha$ is precisely $\{0\}$, i.e., it is injective, as expected.

Suppose now that $\alpha=\sigma+1$. There is an natural diagram of short-exact sequences
$$       \begin{CD}
               0 @>>> Q_\sigma/Q_{\alpha} @>>>       Q/Q_\alpha           @>>> Q/Q_{\sigma} @>>> 0\\
             @.@VV{\gamma}'V   @VV\gamma^\alpha V        @VV\gamma^\sigma V\\
               0 @>>> Q_\sigma/Q_{\alpha} @>>>       Q/Q_\alpha           @>>> Q/Q_{\sigma} @>>> 0\\
      \end{CD}
$$
Let $K'$, $K_\alpha$ and $K_\sigma$ be the kernels of the vertical maps in this diagram, and $C'$, $C_\alpha$ and $C_\sigma$ denote their co-kernels. Since $\gamma^\alpha$ and $\gamma^\sigma$ are surjective, it must be that $C_\alpha=C_\sigma=\{0\}$. Thus, there is a long-exact sequence
$$
0 \to K'\to K_\alpha\to K_\sigma \to C'\to 0.
$$
Furthermore, by induction, $\gamma^\sigma$ is injective, so that $K_\sigma=0$. Consequently, $C'=\{0\}$, so that $\gamma'$ has to be surjective. But since the quotient $Q_\sigma/Q_{\alpha}\cong \Q\otimes B_\alpha$ is finite dimensional over $\Q$, the fact that $\gamma'$ is surjective leads to this that it is also injective. Finally, $K'=\{0\}$, which enables us that $K_\alpha=\{0\}$, as required.
\end{proof}

Particularly, an appeal to Theorem~\ref{4} teaches that, if $G=\oplus_{i\in I} A_i$ is completely decomposable of infinite rank and the types of the summands $A_i$ satisfy the descending type condition, then $G$ is co-finitely injective; but, Corollary~\ref{exclud} manifestly illustrates that a completely decomposable co-finitely surjective group must be divisible. For example, one can just have the types of the subgroups $A_i$ being pairwise incomparable for distinct $A_i, A_j$. So, there are lots of examples besides those mentioned above of a co-finitely injective group (possibly of infinite torsion-free rank) that is {\it not} co-finitely surjective, and hence not co-finitely Hopfian.

\section{Cotorsion groups}

We now deal with cotorsion groups as defined in \cite{F1,F2} and, concretely, when a cotorsion group is co-finitely injective (respectively, co-finitely surjective or co-finitely Hopfian). 

\begin{prop}\label{2.7}
A cotorsion group $G$ is co-finitely surjective if and only if it is divisible.
\end{prop}

\begin{proof}
Sufficiency follows directly from Proposition~\ref{1.9}(b), so assume $G$ is co-finitely surjective. We need to establish that its reduced part is precisely $\{0\}$, so there is no loss of generality in assuming that $G$ is reduced. Consulting with Corollary~\ref{0.25}(b), we can conclude that $G$ is torsion-free. Now, a reduced torsion-free co-torsion group is algebraically compact, so if $G\ne \{0\}$, then, for some prime $p$, $G$ must have a summand isomorphic to the $p$-adic integers, $\hat \Z_p$. However, we already have noted that $\hat \Z_p$ is not co-finitely surjective, and so this contradiction demonstrates that $G=\{0\}$, concluding the proof.
\end{proof}

On the other hand, we have the following characterization of cotorsion groups that are co-finitely injective.

\begin{prop}\label{cotorsion} Let $G$ be a cotorsion group. Then, the following three items are equivalent:

(a) $G$ is co-finitely injective;

(b) $G$ is torsion-free and Hopfian;

(c) $G$ is isomorphic to $\Q^n\oplus (\prod_p G_p)$, where $n<\omega$ and, for each prime $p$, $G_p$ is the direct sum of a finite number of copies of the $p$-adic integers.
\end{prop}

\begin{proof} Consider (a) $\Rightarrow$ (b): In view of Corollary~\ref{0.25}, $G$ must be torsion-free and, moreover, it is clear that it must also be Hopfian.

Turning to (b) $\Rightarrow$ (c): Suppose $G\cong D\oplus R$, where $D$ is divisible and $R$ is reduced. We know $D$ must be torsion-free (since $G$ is), and if it had infinite rank, it would fail to be Hopfian; so $D\cong \Q^n$ for some $n<\omega$. Again, since $R$ is torsion-free and algebraically compact, \cite[Theorem~40.2]{F1} allows us to write that
$$G\cong \prod_p G_p,$$ where, for each prime $p$, $G_p$ is the $p$-adic completion of a free $p$-adic module, say $\hat\Z_p^{({\kappa_p})}$. Suppose, for some prime $p$, that $\kappa_p$ is infinite. It automatically follows that $G_p\cong G_p\oplus \hat\Z_p$, implying that $G\cong G\oplus \hat\Z_p$. Considering the projection of $G$ onto the first summand shows in addition that $G$ is not Hopfian, contrary to assumption.

Consider now (c) $\Rightarrow$ (a): Looking at Proposition~\ref{1.95}, it suffices to assume that $G=\prod_p G_p$ as above. If now $\phi:G\to G$ is an endomorphism, then since each direct component $G_p$ is fully invariant in $G$, we must have $\phi=(\phi_p)_{p\in {\mathcal P}}$, where each $\phi_p:G_p\to G_p$. If $G/\phi(G)$ is finite, it follows that
$$
G/\phi(G)\cong \prod_p (G_p/\phi_p(G_p))
$$
is finite too. So, there is a finite set of primes $\mathcal F$ such that, if $\mathcal Q=\mathcal P\setminus \mathcal F$, then, for every $p\in\mathcal F$, we have $G_p/\phi_p(G_p)$ is finite and, for every $q\in\mathcal Q$, we have $G_q/\phi_q(G_q)=\{0\}$. However, as each $G_p$ is a free $\hat Z_p$-module of finite rank, it follows that, if $p\in \mathcal F$, then $\phi_p$ is injective and, if $q\in \mathcal Q$, then $\phi_q$ is an isomorphism (and hence an injection). This means at once that $\phi=(\phi_p)_{p\in {\mathcal P}}$ must be injective, as required.
\end{proof}

We can now immediately extract the following consequence.

\begin{cor}\label{4.5}
The cotorsion group $G$ is co-finitely Hopfian if and only if it is torsion-free divisible of finite rank.
\end{cor}

So, if $G$ is either finite rank, completely decomposable, or cotorsion, then it is co-finitely Hopfian if, and only if, $G \cong \Q^n$ for some $n<\omega$. In what follows, we want to establish that this does {\it not} hold for arbitrary groups of infinite rank.

\section{Examples of Butler groups of infinite rank}

Suppose $G$ is a co-finitely Hopfian group; so, in particular, it is torsion-free. Then, we can infer that $G\cong \Q^n$ for some $n<\omega$ in the following cases: (a) $G$ has finite rank (Proposition~\ref{3}); (b) $G$ is completely decomposable (Corollary~\ref{exclud}); (c) $G$ is cotorsion (Corollary~4.5). In this section, we show that for some other types of groups, there are, in fact, many examples of reduced co-finitely Hopfian groups.

Recalling now \cite{B}, a torsion-free group $B$ of finite rank is said to be a {\it Butler group} if it is a pure subgroup of a completely decomposable group of finite rank; or, equivalently, it is an epimorphic image of a completely decomposable group of finite rank. There are two non-trivial generalizations of the notion of a Butler group to the infinite rank case (see, for a more account, \cite{F}): the group $B$ is said to be a {\it $B_1$-group} if ${\rm Pext}^1(B,T)=\{0\}$ for all torsion groups $T$, and is said to be a {\it $B_2$-group} if it is the smoothly ascending chain of pure subgroups $A_i$ such that, for each $i$, $A_{i+1}=A_i+C_i$, where $C_i$ is a pure finite-rank Butler subgroup of $B$. For groups of arbitrary rank, a $B_2$-group will always be a $B_1$-group and, for countable ranks, the two notions coincide.

Notice that, in \cite{DK1}, the authors defined the group $G$ to have {\it finite injective rank} if an endomorphism $\phi:G\to G$ is injective if, {\bf and only if,} $G/\phi(G)$ is finite. In particular, if $G$ has finite injective rank, then it is trivially co-finitely injective, and hence torsion-free.

\medskip

The next construction sheds a bit more clarity in this light.

\begin{ex}\label{infin}
There is a reduced (so, in particular, {\it not} divisible) torsion-free $B_2$-group $G$ of countably infinite rank that is co-finitely Hopfian (and hence co-finitely injective, and thus Hopfian), but $G$ does {\it not} have finite injective rank.
\end{ex}

\begin{proof}
Let $\mathcal P=\{p_0, p_1, p_2, \dots\}$ and $\mathcal Q=\{q_1, q_2, \dots\}$ be a partition of the collection of all primes into two infinite sets. Also, let $\{{\bf e}_i\}$ be the ``standard basis" for $V=\Q^{(\omega)}$. If ${\bf v}\in V$ and $p$ is a prime, let $(1/p^\infty){\bf v}=\langle (1/p^k){\bf v}: k\in \N\rangle$. We, thereby, define the group
$$
G:= \left(\sum_{i\in \omega} (1/p_i^{\infty})\,{\bf e}_i\right ) +  \left(\sum_{i\in \N} (1/q_i^{\infty})\,({\bf e}_0+{\bf e}_{i})\right).
$$

Note that, setting for each $i<\omega$
$$
A_i:=\langle {\bf e}_0, {\bf e}_1,\dots, {\bf e}_i \rangle_* \ \ \ {\rm and}\ \ \ C_i:=(1/p_{i+1}^\infty)\,{\bf e}_{i+1} + (1/q_{i+1}^\infty)\,({\bf e}_0+{\bf e}_{i+1}),
$$
trivially shows that $G$ is, in fact, a $B_2$-group. (In fact, it is straightforward to verify that it can be embedded as a pure subgroup of a countable rank completely decomposable group, which we leave out to the interested reader for a direct check.)

Let $\phi:G\to G$ be an endomorphism. We claim that, for some $n\in \Z$, $\phi(x)=nx$ for all $x\in G$: Indeed, for each $i< \omega$, $\langle p^\infty_i{\bf e}_i\rangle_*=p^\infty_i G$ is obviously fully invariant in $G$, so that $\phi({\bf e_i})=n_ip_i^{k_i}{\bf e}_i$, where $n_i, k_i\in \Z$. Similarly, each $(1/q_i^{\infty})\,({\bf e}_0+{\bf e}_{i})=q_i^\infty G$ is fully invariant in $G$, so that
$$
\phi({\bf e}_0+{\bf e}_{i})=m_iq_i^{\ell_i}({\bf e}_0+{\bf e}_{i})=m_iq_i^{\ell_i}{\bf e}_0+m_iq_i^{\ell_i}{\bf e}_{i},
$$ where $m_i, \ell_i\in \Z$. It follows now, for all $i\in \N$, that
$$
              n_0p_0^{k_0}{\bf e}_0+n_{i}p_{i}^{k_{i}}{\bf e}_{i}=\phi({\bf e}_0)+\phi({\bf e}_{i})=\phi({\bf e}_0+{\bf e}_{i})=m_iq_i^{\ell_i}{\bf e}_0+m_iq_i^{\ell_i}{\bf e}_{i}.
$$
This is only possible if, for all $i\in \N$, we have
$$n_ip_i^{k_i}=m_{i}q_{i}^{\ell_{i}}=n_0p_0^{k_0}=:n\in \Z,$$
giving our claim.

We now show that, if $\phi(x)=nx$ and $n\ne \pm 1$, then $G/\phi(G)$ is infinite: Certainly, this is the case if $n=0$, so we may freely assume $n\ne 0$. Thus, assuming $n\ne 0$ is composite, let $p$ be some prime dividing $n$, say to the power $\ell$. If $L_p$ denotes the integers localized at $p$, then localizing at $p$ gives the following relations
$$
         G_{(p)}= G\otimes L_p \cong \Q\oplus L_p^{(\N)}.
$$
It, likewise, follows that
$$
(G/\phi(G))_{(p)}\cong G_{(p)}/nG_{(p)}\cong (\Q/p^\ell\Q)\oplus (L_p/p^\ell L_p)^{(\N)}\cong \Z(p^\ell)^{(\N)},
$$
which is apparently infinite. Therefore, the only way $G/\phi(G)$ can be finite is when $n=\pm 1$, so that $\phi$ is an isomorphism, i.e., $G$ is co-finitely Hopfian, as pursued.

On the other hand, if $\phi$ is multiplication by some prime $p$, then $\phi$ is injective but $G/\phi(G)=G/pG$ is infinite, so that $G$ does not have finite injective rank, as we expected.
\end{proof}

In addition, the above example can be expanded to produce a series of examples of reduced co-finitely Hopfian groups of rank up to $2^{(2^{\aleph_0})}=2^c$: In fact, using the notation in that example, suppose that $\sigma:\N\to \N'$ is a permutation of $\N$ that is not the identity, so that $j':=\sigma(j)\ne j$ for some $j\in \N$. Doing the analogous construction by replacing $(1/q_i^{\infty})\,({\bf e}_0+{\bf e}_{i})$ with $(1/q_{\sigma (i)}^{\infty})\,({\bf e}_0+{\bf e}_{i})$ for all $i\in \N$, we will obtain a very similar group which we designate by $G_\sigma$. Note that, if $\phi:G\to G_\sigma$ is any homomorphism, then, for all $i\in \omega$, $\phi(p^\infty_iG)\subseteq p^\infty_iG_\sigma$ allows to set that $\phi({\bf e}_i)=n_ip_i^{k_i}{\bf e}_i$, where $n_i, k_i\in \Z$. If $j\ne \sigma(j)$, then $\phi(q^\infty_j G)\subseteq q^\infty_jG_\sigma$ readily gives that $\phi({\bf e}(0))=0$, which, in turn, implies that $\phi$ is the zero map. This computation verifies that
$
       \{ G_\sigma: \sigma\in {\mathcal S}\}
$
(where $\mathcal S$ is the collection of all bijective maps $\sigma:\N\to \N$, and $G_{1_{\N}}=G$ is a so-called {\it rigid system}). It now easily follows that $H_{\rm sum}:=\bigoplus_{\sigma\in \mathcal S} G_\sigma$ and $H_{\rm prod}:=\prod_{\sigma\in \mathcal S} G_\sigma$ have endomorphism rings isomorphic to $\prod_{\sigma \in \mathcal S} \Z$, and thus are, in fact, co-finitely Hopfian. Clearly, $H_{\rm prod}$ has now the desired rank of $2^c$.

\medskip

In fact, using more sophisticated existence results, it is straightforward to generalize these constructions to produce reduced co-finitely Hopfian groups of arbitrarily large rank: Indeed, it is well known that there are arbitrarily large torsion-free reduced groups $G$ with endomorphism ring isomorphic to $\Z$ such that, for all primes $p$, $G/pG$ is infinite (cf. \cite{F1,F2}), and an identical proof to the above shows that each such will be co-finitely Hopfian, but will {\it not} have finite injective rank. In fact, large rigid systems of such groups can also be constructed, same as above, and their direct sums and products will once again be co-finitely Hopfian.

\medskip

Again, suppose $G$ is reduced and co-finitely surjective. If $G$ has finite rank or is either completely decomposable or cotorsion, then it must be $\{0\}$, so that it is also co-finitely injective (whence co-finitely Hopfian). One may ask, then, at least for $B_2$-groups, is every co-finitely surjective group also co-finitely injective, and hence co-finitely Hopfian? The following example provides a negative answer to this question.

\begin{ex}\label{7.0}
There is a reduced torsion-free $B_2$-group $H$ of countably infinite rank that is co-finitely surjective, but {\it not} Hopfian. So, this $H$ is {\it neither} co-finitely injective, {\it nor} co-finitely Hopfian, {\it nor} does it have finite injective rank.
\end{ex}

\begin{proof}
Let $G$ be the group from Example~\ref{infin}, let $I$ be some (non-empty) index set, and let $H:=G^{(I)}=\bigoplus_{i\in I} G_i$, where each $G_i\cong G$. Now, it is obvious that $H$ has to be a $B_2$-group; in fact, since $G$ is a pure subgroup of a completely decomposable group, so is $H$.

We next propose to prove the following three properties:

\medskip

(A) $H$ does {\it not} have finite injective rank for any index set $I$.

(B) $H$ is co-finitely injective if and only if it is Hopfian; if and only if $I$ is finite.

(C) $H$ is co-finitely surjective for any index set $I$.

\medskip

Since $G$ does not have finite injective rank and is isomorphic to a summand of $H$, one sees that $H$ never has finite injective rank, so that (A) holds. Before verifying the other two assertions, we discuss some general ideas.

Since $G$ is torsion-free, if
$$
          0 \to A\to B\to C\to 0
$$
is a short-exact sequence, then, as the torsion product ${\rm Tor} (G,C)=0$, it follows that there is an induced short-exact sequence
$$
               0 \to A\otimes G\to B\otimes G\to C\otimes G\to 0.
$$
(In other words, as a $\Z$-module, the group $G$ is torsion-free and hence {\it flat}.) In particular, a surjective (respectively, injective) homomorphism $\gamma: X\to Y$ will induce another surjective (respectively, injective) homomorphism $\gamma\otimes G:X\otimes G\to Y\otimes G$.

Besides, if $F=\bigoplus_{i\in I} \Z {\bf e}_i$ is a free group of rank $\vert I\vert$, then we can identify $H$ with $F\otimes  G$.

Supposing $\phi:H\to H$ is an endomorphism, we claim that there is a homomorphism $\gamma:F\to F$ such that $\phi=\gamma\otimes G$: in fact, if $i,j\in I$, let $\kappa_i:G_i\to H$ be the natural injection and let $\pi_j:H\to G_j$ be the natural surjection. It follows, for all $i,j$, that $\pi_j\circ \phi\circ \kappa_i$ will be multiplication by some $\mu_{i,j}\in \Z$. If $0\ne a\in G$, then
$$
              \phi({\bf e}_i\otimes a)=\sum_{j\in I}\mu_{i,j} {\bf e}_j\otimes a\in H.
$$
It thus follows that, for all $i$, $\mu_{i,j}=0$ for all but finitely many $j$. Therefore, setting
$$
                   \gamma({\bf e}_i):=\sum_{j\in I}\mu_{i,j} {\bf e}_j\in F,
$$
we perceive that $\gamma:F\to F$ and $\phi=\gamma\otimes G$.

Turning to (B), certainly, if $H$ is co-finitely injective, then it must be Hopfian, and it is also a routine observation that, if $H$ is Hopfian, then $I$ must be finite.

Let $\phi:H\to H$ and $\gamma:F\to F$ be such that $\phi=\gamma\otimes G$; so, again by the above discussion of exactness, we can conclude that $\phi$ is injective or surjective if, and only if, the same holds for $\gamma$. To finish off the proof of (B), therefore, we need to prove that, if $I$ is finite and $\gamma$ is not injective, then $H/\phi(H)$ is infinite. Observe that when $I$ is finite, if $\gamma:F\to F$ is not injective, then it is not surjective too. Consequently, to complete the proof of (B), as well as to prove (C), we need to show that, if $\gamma$ is not surjective, then $H/\phi(H)$ is infinite.

So, to that target, assume $\phi$ and $\gamma$ fail to be surjective. Again by exactness, we can identify $\gamma(F)\otimes G$ with $\phi(H)$, and $H/\phi(H)$ with $(F/\gamma(F))\otimes G$.

Since $H/\phi(H)\ne \overline{\{0\}}$, it follows that $F/\gamma(F)\ne \overline{\{0\}}$. Let $\overline{0}\ne b\in F/\gamma(F)$ and consider $B=\langle b\rangle$. Since $B\leq F/\gamma(F)$, we can view $B\otimes G$ as a subgroup of $H/\phi(H)$.

Note that, if $B\cong \Z$, then $B\otimes G\cong G$ is infinite. On the other hand, if $B\cong \Z(n)$ for some $1<n\in \N$, then $B\otimes G\cong G/nG$, which in the proof of Example~\ref{infin} we saw must be infinite. In either case, since $B\otimes G$ can be identified with a subgroup of $H/\phi(H)$, we can conclude that $H/\phi(H)$ is infinite, which is what we needed to show.
\end{proof}

So, again, if $G$ is from Example~\ref{infin}, then $H=G^{(\omega)}$ will be co-finitely surjective, but {\it not} co-finitely injective, whereas if $n<\omega$, then $H=G^{(n)}$ will be another example of a reduced co-finitely Hopfian group.

\medskip  

The group in the next statement was constructed in \cite{AR} to provide an example of a $B_2$-group of countable rank that failed to be a pure subgroup of a completely decomposable torsion-free group. Thus, we arrive at the following.

\begin{ex}\label{7.0}
There is a countable reduced torsion-free $B_2$-group $R$ that has finite injective rank (and hence is co-finitely injective), but is {\it not} co-finitely surjective (and hence is not co-finitely Hopfian). In fact, one can be found that fails to be a pure subgroup of a completely decomposable group.
\end{ex}

It is worth mentioning that in \cite{DK1} this construction was analyzed and, in particular, in  \cite[Proposition~2.14]{DK1} it was verified that $R$ has finite injective rank. In that computation, it was shown, for instance, that if $\phi$ is the monomorphism given by multiplication by $2$, then the factor-group $R/\phi(R)\cong \Z(2)$ is finite, but $\phi$ is {\it not} surjective. Therefore, $R$ is {\it not} co-finitely surjective, as promised.

\section{Almost co-finitely Hopfian groups}\label{almost}

We start our work here with a slight generalization of the notion of co-finitely Hopfian groups defined above.

\begin{defn}\label{almcof} The group $G$ is {\it almost co-finitely Hopfian} if, whenever $\phi:G\to G$ is an endomorphism with finite co-kernel $G/\phi(G)$, then $\phi$ also has finite kernel.
\end{defn}

It is pretty clear that any co-finitely injective group is almost co-finitely Hopfian, but any non-zero finite group is co-finitely Hopfian, but not co-finitely injective.

\medskip

The following statement gives us some helpful preliminary information about strongly finitely Hopfian torsion groups.

\begin{prop}\label{torsion}
Suppose $T=\bigoplus_p T_p$ is a torsion group. Then, $T$ is almost co-finitely Hopfian if and only if each $T_p$ has that property and all but finitely many $T_p$ are Hopfian.
\end{prop}

\begin{proof}
Considering sufficiency, suppose $\phi:T\to T$ is an endomorphism with a finite co-kernel. We, respectively, denote the kernel and co-kernel of $\phi$ by $K$ and $C$. Since $C\cong \bigoplus_p C_{T_p}$ is finite for almost all $p$, we can derive $C_{T_p}=0$, i.e., $\phi_{T_p}$ is surjective. But since all but finitely many of the $T_p$ are Hopfian, we may deduce that $K_{T_p}=\{0\}$ for all but finitely many $p$. And since each $C_{T_p}$ must be finite, we detect that all of the terms in $K\cong \bigoplus_p K_{T_p}$ are finite and all but finitely many are $\{0\}$; that is, $K$ must be finite, so that $T$ is almost co-finitely Hopfian, as wanted.

Regarding necessity, since it is easily checked that a summand of an almost co-finitely Hopfian group retains that property, we need only show that all but finitely many $T_p$ are actually Hopfian. To that goal, we just suppose the contrary that an infinite number of these are not such, and show that $T$ is not almost co-finitely Hopfian. In fact, as already observed, reducing to a summand, we may assume that, for every prime $p$, there is a surjective homomorphism $T_p\to T_p$ that is not injective. Hence, splicing these together to get an endomorphism $\phi:T\to T$ that is onto, so that $C=\{0\}$ but such that $K\cong \bigoplus_p K_{T_p}$ is infinite, we are done, as desired.
\end{proof}

The following assertion unambiguously discovers that the torsion-free almost co-finitely Hopfian groups are {\it not} new to us.

\begin{prop}\label{TF}
A torsion-free group $G$ is almost co-finitely Hopfian if and only if it is co-finitely injective.
\end{prop}

\begin{proof} Sufficiency being quite elementary, suppose $G$ is almost co-finitely Hopfian. Let $\phi:G\to G$ be an endomorphism with finite co-kernel; thus, we need to show that it is injective. However, since $G$ is almost co-finitely Hopfian, its kernel must be a finite subgroups of $G$, and since $G$ is torsion-free, all of its finite subgroups are exactly $\{0\}$, as required. 
\end{proof}

\begin{cor}\label{torsion-free}
A torsion-free group of finite rank will always be almost co-finitely Hopfian.
\end{cor}

The next reduction claim is rather useful.

\begin{prop}\label{initial}
Suppose $G$ is a group. If both $T$ and $G/T$ are almost co-finitely Hopfian groups, then $G$ is too almost co-finitely Hopfian.
\end{prop}

\begin{proof} Suppose $\phi:G\to G$ has finite co-kernel; what we just want to prove is that $\phi$ also has finite kernel. To that end, if we set $Q:=G/T$, then $\phi$ determines two endomorphisms $\phi_T:T\to T$ and $\phi_Q: Q\to Q$. Thus, there is a resulting diagram:
$$       \begin{CD}
               0 @>>> T @>>>       G           @>>> Q @>>> 0\\
             @.@VV{\phi_T}V   @VV\phi V        @VV\phi_Q V\\
               0 @>>> T @>>>       G           @>>> Q @>>> 0\\
      \end{CD}
$$
If $\phi_T$, $\phi$ and $\phi_Q$ have kernels $K_T$, $K_G$ and $K_Q$, and co-kernels $C_T$, $C_G$ and $C_Q$, respectively, then there is an exact sequence
$$
   0\to  K_T\to K_G\to    K_Q \to C_T\to C_G\to C_Q\to 0.
$$
As we are assuming that $C_G$ is finite, it follows at once that $C_Q$ is also finite. Next, since $Q$ is almost co-finitely Hopfian, thanks to Proposition~\ref{TF}, it must be co-finitely injective, i.e.,  $K_Q=\{0\}$.

Our sequence then guarantees that $C_T\to C_G$ is injective and, in particular, since $C_G$ is finite, so is $C_T$. Since we are also assuming $T$ is almost co-finitely Hopfian, we can conclude that $K_T$ is finite. But, since by what we have already shown above $K_Q=\{0\}$, our sequence then insures that $K_G\cong K_T$ is finite, as required.
\end{proof}

The next two necessary and sufficient conditions somewhat give a partial characterization of almost co-finitely Hopfian groups in certain situations.

\begin{prop}\label{later}
Suppose $G$ is a group such that $G/T$ is almost co-finitely Hopfian. If $T$ has bounded $p$-torsion (i.e., each $T_p$ is bounded for all primes $p$), then $G$ is almost co-finitely Hopfian if and only if each $T_p$ is finite.
\end{prop}

\begin{proof}
Suppose first we act on necessity assuming the contrary that some $T_p$ is infinite. Since $T_p$ is bounded, it is plainly verified that we have a direct decomposition $G\cong H\oplus \Z(p^n)^{(\N)}$. There is, clearly, a surjective homomorphism $$\phi:\Z(p^n)^{(\N)}\to \Z(p^n)^{(\N)}$$ whose kernel is infinite. Extending $\phi$ to $G$ by setting it equal to the identity on $H$, then one inspects that $\phi$ is onto, so that its co-kernel is obviously $\{0\}$. On the other hand, the kernel of $\phi$ is infinite, so that $G$ is not almost co-finitely Hopfian, as asked for.

Conversely, acting on sufficiency, suppose each $T_p$ is finite. Since a finite group is trivially both Hopfian and almost co-finitely Hopfian, we apply Proposition~\ref{torsion} to infer that $T$ is almost co-finitely Hopfian.  Therefore, according to Proposition~\ref{initial}, the group $G$ is so as well, as claimed.
\end{proof}

The next criterion follows immediately from Propositions~\ref{torsion-free}, \ref{later} and the total characterization of Bassian groups given in \cite{CDG} (see \cite{DK} too). Recall that a group is termed {\it Bassian} if it cannot be embedded in a proper homomorphic image of itself.

\begin{cor}\label{Bassian}
Suppose $G$ has finite torsion-free rank and bounded $p$-torsion. Then, $G$ is almost co-finitely Hopfian if and only if it is Bassian.
\end{cor}

\section{Almost finitely Hopfian groups}

We proceed here by stating the following concept.

\begin{defn}\label{almcof} The group $G$ is {\it almost finitely Hopfian} if, whenever $\phi:G\to G$ is a surjective endomorphism, then $\phi$ has finite kernel.
\end{defn}

Apparently, any almost co-finitely Hopfian and in addition, any Hopfian group, is almost finitely Hopfian.

Note also that this definition is reasonable as the fact that $G/\phi(G)$ is finite does {\it not} imply that ker$\phi$ is finite too. Indeed, whenever $G\cong G\oplus K$ for some infinite $K$, then if $\phi$ is the projection onto the first summand, then the quotient $G/\phi(G)=\overline{\{0\}}$ is finite, whereas $\mathrm{ker}(\phi)=K$ is infinite.

\medskip

The next assertion shows that the converse holds for most commonly encountered groups.

\begin{thm}\label{sep} Suppose $G$ is a group such that, for all primes $p$, $T_p$ is separable. Then, $G$ is almost finitely Hopfian if and only if it is Hopfian.
\end{thm}

\begin{proof}
Sufficiency again being obvious, suppose $G$ is almost finitely Hopfian. If $\phi:G\to G$ is an epimorphism, then we need to prove its kernel $K$ is precisely $\{0\}$.

Since $G$ is almost finitely Hopfian, we can conclude $K$ is finite, so that $K\leq T$. Therefore, one may decompose $K=\bigoplus_p K_p$, where each $K_p$ is a finite subgroup of $T_p$. If $p$ is an arbitrary prime, it just suffices to establish that $K_p=\{0\}$.

Since $T_p$ is separable, there is an natural $n\in \N$ such that $(p^n G)\cap  K_p=\{0\}$. It follows that $G\cong H_p\oplus B_p$, where $B_p$ is a $p^n$-high subgroup containing $K_p$. If $k\in \N$, it is evident that $\Z(p^k)^{(\N)}$ is not co-finitely injective, so that $G$ cannot have such a summand. Consequently, $B$ too cannot have such a summand, and since $p^n B_p=\{0\}$, it follows that $B_p$ is actually finite, as is $K_p$.

Setting $G':=G/K$, so $G'\cong G$. We also let $$T'=T/K\cong \bigoplus_p (T_p/B_p),$$ which is the torsion subgroup of $G'$. Since $G\cong G'$, we have $T\cong T'$. Note that $T_p=S_p\oplus B_p$, where $S_p=T_p\cap H_p$. It also routinely follows that $$T'_p\cong T_p/K_p\cong S_p\oplus (B_p/K_p)$$ can be identified with the $p$-torsion subgroup of $G'$. And, if we let $B_p'=B_p/K_p$, then it must be that $$p^n B_p'=[p^n B_p+K_p]/K=\overline{\{0\}}.$$ A simple check now gives that $(p^n T')[p]\cong S_p[p]$, so that $B_p'$ is also $p^n$-high in $T'$. Therefore, $B\cong B'$. But since this is finite and $B_p'\cong B_p/K_p$, we can infer that $K=\{0\}$. Finally, $\phi$ is injective, so that $T$ is Hopfian, as asserted.
\end{proof}

\section{Concluding Discussion and Open Problems}

Mimicking \cite{A}, a group is said to be {\it strongly homogeneous} if, for any two pure rank 1 subgroups, there is an automorphism sending one onto the other. The finite rank case of torsion-free strongly homogeneous groups is characterized there. We, however, are interested only in the infinite torsion-free rank situation as in the case of finite torsion-free rank our establishments above completely settle all we need. To be more concrete, we would like to discover some close relationship between strongly homogeneous (torsion-free) groups and the variations of Hopficity as defined above. To that aim, we first observe that it is extremely difficult to decide whether or not (reduced) strongly homogeneous torsion-free groups of infinite rank are co-finitely injective.

\medskip

We end our work with the following two challenging questions of some interest and importance.

\medskip

The first problem is related to Definition~\ref{almcof} and the results established in Subsection~\ref{almost}.

\begin{prob}\label{8} Characterize almost co-finitely Hopfian groups, that are those groups $G$ such that whenever $\phi: G\to G$ is an endomorphism such that $G/\phi(G)$ is finite, then ker$\phi$ is also finite.
\end{prob}

The class of groups in the next second problem generalizes the classical Hopfian groups and is relevant to the obtained above Theorem~\ref{sep}.

\begin{prob}\label{9} Characterize almost finitely Hopfian groups, that are those groups $G$ such that whenever $\phi: G\to G$ is a surjective endomorphism, then ker$\phi$ is finite.
\end{prob}

It is worthwhile to mention that these two logically raised problems are {\it not} unsensible, because they are reciprocal to \cite[Theorem 5.1]{ER}. 

\medskip

We terminate our work with the following problem which, as we have seen above, has an affirmative answer for groups of finite rank, as well as for group which are either completely decomposable or cotorsion.

\begin{prob}\label{easystatement}
It is clear that any co-finitely injective group is torsion-free and Hopfian. Does the converse hold, that is, if a group is torsion-free and Hopfian, must it necessarily be co-finitely injective?
\end{prob}



\vskip2pc

\begin{thebibliography}{99}

\bibitem{A}
D.M. Arnold, {\it Strongly homogeneous torsion free Abelian groups of finite rank}, Proc. Am. Math. Soc. \textbf{56} (1976), 67--72.

\bibitem{AR}
D.M. Arnold and K.M. Rangaswamy, {\it A note on countable Butler groups}, Boll. Unione Mat. Ital. (Sez. B - Artic. Ric. Mat.) Ser. 8 \textbf{10} (2007), no. 3, 605--611.

\bibitem{BGHM}
M.R. Bridson, D. Groves, J.A. Hillman and G.J. Martin, {\it Cofinitely Hopfian groups, open mappings and knot complements}, Groups Geom. Dyn. \textbf{4} (2010), no. 4, 693--707.

\bibitem{B}
M.C.R. Butler, {\it A class of torsion-free abelian groups of finite rank}, Proc. London Math. Soc. \textbf{15} (1965), 680--698.

\bibitem{CDG}
A.R. Chekhlov, P.V. Danchev and B. Goldsmith, {\it On the Bassian property for Abelian groups}, Arch. Math. (Basel) \textbf{117} (2021), no. 6, 593--600.

\bibitem{DK}
P.V. Danchev and P.W. Keef, {\it Generalized Bassian and other mixed Abelian groups with bounded $p$-torsion}, J. Algebra \textbf{663} (2025), 1--19.

\bibitem{DK1}
P.V. Danchev and P.W. Keef, {\it Mono Dedekind-finite Abelian groups}, Math. Notes {\bf 118} (2025), nos. 5-6.

\bibitem{ER}
G. Endimioni and D.J.S. Robinson, {\it On co-Hopfian groups}, Publ. Math. Debrecen \textbf{67} (2005), nos. 3-4, 423--436.

\bibitem{F1}
L. Fuchs, Infinite Abelian Groups, Vols. {\bf I} \& {\bf II}, Acad. Press, New York and London (1970 \& 1973).

\bibitem{F}
L. Fuchs, {\it Butler groups of infinite rank}, J. Pure Appl. Algebra \textbf{98} (1995), 25--44.

\bibitem{F2}
L. Fuchs, Abelian Groups, Springer Monographs in Math., Springer Internat. Publ., Switzerland (2015).

\bibitem{H}
H. Hopf, {\it Beitr\"age zur Klassizierung der Fl\"achen Abbildungen}, J. Reine Angew. Math. \textbf{165} (1931), 225--236.

\end{thebibliography}
\end{document}